\title{Double integrals on a weighted projective plane and the Hilbert modular functions for $\mathbb{Q}(\sqrt{5})$}
\author{Atsuhira Nagano}
\documentclass{article}
\usepackage{mathrsfs,bm,graphics,amsfonts,amsthm,amssymb,amsmath,amscd,booktabs}
\usepackage[pdflatex]{graphicx}
\def\bigzerou{\smash{\lower1.7ex\hbox{\b 0}}}

\newtheorem{thm}{Theorem}[section]
\newtheorem{defn}{Definition}[section]
\newtheorem{lem}{Lemma}[section]
\newtheorem{prop}{Proposition}[section]
\newtheorem{rem}{Remark}[section]

\newtheorem{cor}{Corollary}[section]

\def\comment#1{{ }}
\makeatletter
  
      \@addtoreset{equation}{section}
        
      \makeatletter

\begin{document}

\maketitle      

\begin{abstract}
The aim of this paper is
to give an explicit extension of the classical elliptic integrals to the Hilbert modular case for $\mathbb{Q}(\sqrt{5})$.
We study a family of  Kummer surfaces corresponding to the Humbert surface of invariant $5$ with two complex parameters.
Our Kummer surface is given by a double covering of the weighted projective space $\mathbb{P}(1:1:2)$
branched along a parabola and a quintic curve.
The period mapping for our family 
 is given by  double integrals of an algebraic function  on chambers coming from an arrangement of a parabola and a quintic curve in $\mathbb{C}^2$.
\end{abstract}




\footnote[0]{Keywords:  Elliptic integrals ; Hilbert modular functions ; Kummer surfaces  }
\footnote[0]{Mathematics Subject Classification 2010:  11F46, 14J28}
\footnote[0]{Running head: Double integrals and Hilbert modular  functions}
\setlength{\baselineskip}{14 pt}

\section*{Introduction}

The aim of this paper is to give a canonical extension of the classical elliptic integrals to the Hilbert modular case for $\mathbb{Q}(\sqrt{5})$.

The arrangement of $4$ points on the projective line $\mathbb{P}^1(\mathbb{C})$ is deeply related to the elliptic modular functions for the principal congruence subgroup $\Gamma(2)$.
The double covering of  $\mathbb{P}^1(\mathbb{C})$ branched  at $4$ points gives an elliptic curve.
The coordinate of the configuration space of  $4$ branch points on $\mathbb{P}^1(\mathbb{C})$
gives a modular function for $\Gamma(2)$
via the period mapping of the family of the corresponding elliptic curves.
 
 One of the most successful  extensions of the above classical story  to several variables is given by K. Matsumoto, T. Sasaki and M. Yoshida \cite{MSY}.
They showed an interesting relation 
between the arrangement of $6$ lines on the projective plane $\mathbb{P}^2(\mathbb{C})$
and the modular functions on a $4$ dimensional bounded symmetric space of type $I$ 
via the period mapping of the family of $K3$ surfaces  coming from the arrangement of $6$ lines.

We shall give another natural extension of the classical elliptic integrals to a case of several variables.
The Hilbert modular functions for  real quadratic fields  are very popular among modular functions of several variables.
However, to the best of the author's knowledge,
to obtain  simple and geometric extensions of the classical elliptic integrals to the Hilbert modular cases is a highly non-trivial problem. 
Although the Hilbert modular functions with level $2$ structure  can be obtained from the moduli of hyperelliptic curves of genus $2$,
they are characterized by a complicated modular equations (see Remark \ref{HumbertRemark}).

In this paper,  we focus on the Hilbert modular functions for $\mathbb{Q}(\sqrt{5})$.
Since the real quadratic field $\mathbb{Q}(\sqrt{5})$ gives the smallest discriminant, 
several researchers (for example, K. B. Gundlach \cite{Gundlach}, F. Hirzebruch \cite{Hirzebruch}, R. M\"uller \cite{Muller}) studied this case in detail.
We shall give a simple and geometric interpretation of the Hilbert modular functions in this case.
We consider the double integrals of the algebraic function $F$ in (\ref{alg.F})
of $2$ variables on chambers surrounded by the parabola $P$ in (\ref{P}) and the quintic curve $Q$ in (\ref{Q}) with the $(2,5)$-cusp.
These double integrals are equal to the period integrals of the Kummer surface $K(X,Y)$ in (\ref{K(X,Y)}). 
The equation   (\ref{K(X,Y)}) gives a double covering of the weighted projective plane $\mathbb{P}(1:1:2)$
branched along $P$ and $Q$
and the complex parameters $(X,Y)$ determine the arrangement of the branch loci.
The parameters  $(X,Y)$ are regarded as a pair of the Hilbert modular functions for $\mathbb{Q}(\sqrt{5})$
 via the explicit double integrals (see  Remark \ref{StoK} and Theorem \ref{DoubleIntegralThm}). 
Our results  are  coherent  with the story of the classical elliptic integrals (see Table 1).
The results in this paper are used
in the  paper \cite{NaganoShiga}.

 \begin{table}[h]\label{w-Table}
\center
\begin{tabular}{ccc}
\toprule
&Classical Story &  Result of This Paper \\
\midrule 
Base Space & $\mathbb{P}^1(\mathbb{C})$ & $\mathbb{P}(1:1:2)$ \\
Branch Loci & $4$ points & $P$ and $Q$\\
Variety & Elliptic curve & Kummer surface $K(X,Y)$ \\
Arrangement  & Elliptic modular function for $\Gamma(2)$ & Hilbert modular functions   for $\mathbb{Q}(\sqrt{5})$ \\
\bottomrule
\end{tabular}
\caption{The classical elliptic integrals and the result of this paper.} 
\end{table}

The author conjectures that 
we can similarly obtain simple and geometric interpretations of other Hilbert modular functions also,
using suitable weighted projective planes.
Our results might give a first step of such an approach to Hilbert modular functions.

\section{The Kummer surface $K(X,Y)$ and the Hilbert modular functions for $\mathbb{Q}(\sqrt{5})$}

We consider the period mapping for the family $\mathcal{K}=\{K(X,Y)\}$ of surfaces  where
\begin{eqnarray}\label{K(X,Y)}
K(X,Y):  v^2=(u^2-2y^5)(u-(5y^2-10X y+Y))            
\end{eqnarray}
for $(X,Y)\not=(0,0)$.
The equation (\ref{K(X,Y)}) gives a double covering of the $(y,u)$-space 
branched along
the parabola 
\begin{align}\label{P}
u=5 y^2 -10 X y +Y
\end{align}
and the quintic curve 
\begin{align}\label{Q}
u^2 = 2 y^5
\end{align}
with the $(2,5)$-cusp
$(y,u)=(0,0)$.
The parameters $(X,Y)$ define the arrangement of the divisors $P$ and $Q$. 
In this section, we see the properties of  the family $\mathcal{K}$.

\subsection{The Hilbert modular functions for $\mathbb{Q}(\sqrt{5})$ and the $K3$ surface $S(X,Y)$}

In this subsection, we survey the results of  \cite{Nagano}.

Let $\mathcal{O}$ be the ring of integers in the real quadratic field $\mathbb{Q}(\sqrt{5})$.
Set $\mathbb{H}=\{z\in \mathbb{C}| {\rm Im}( z) >0\}.$ The Hilbert modular group $PSL(2,\mathcal{O})$ acts on 
$\mathbb{H}\times\mathbb{H}$ by 
\begin{eqnarray*}
\begin{pmatrix}
\alpha &\beta \\
\gamma &\delta 
\end{pmatrix}
:
(z _1, z_2)\mapsto 
\Big(\frac{\alpha z_1 +\beta}{\gamma z_1 +\delta} ,\frac{\alpha ' z_2+\beta'}{\gamma' z_2+\delta' }\Big),
\end{eqnarray*}
for 
$
g = \begin{pmatrix}\alpha&\beta\\ \gamma&\delta\end{pmatrix} \in PSL(2,\mathcal{O}),
$
where $'$ means the conjugate in $\mathbb{Q}(\sqrt{5})$.
We consider the involution
$
\tau:(z_1,z_2)\mapsto(z_2,z_1)
$
also.

\begin{defn}
If a holomorphic function $g$ on $\mathbb{H}\times\mathbb{H}$
satisfies the transformation law
$$
g\Big(\frac{\alpha z_1+\beta}{\gamma z_1+\delta},\frac{\alpha' z_2+\beta'}{\gamma' z_2+\delta'} \Big)= (\gamma  z_1+\delta)^k (\gamma'z_2+\delta')^k g(z_1,z_2)
$$ 
for any $\displaystyle \begin{pmatrix} \alpha&\beta\\ \gamma&\delta \end{pmatrix}\in PSL(2,\mathcal{O})$,
we call g a Hilbert modular form of weight $k$ for $\mathbb{Q}(\sqrt{5})$.
If $g(z_2,z_1)=g(z_1,z_2)$,  $g$ is called a symmetric modular form.

If a meromorphic function $f$ on $\mathbb{H}\times \mathbb{H}$ satisfies
$$
f\Big(\frac{\alpha z_1+\beta}{\gamma z_1+\delta},\frac{\alpha'z_2+\beta'}{\gamma'z_2+\delta'}\Big)= f(z_1,z_2)
$$
for any $\displaystyle \begin{pmatrix} \alpha&\beta\\ \gamma&\delta\end{pmatrix}\in PSL(2,\mathcal{O})$,
we call $f$  a Hilbert modular function for $\mathbb{Q}(\sqrt{5})$.
\end{defn}

\begin{rem}\label{RemarkHilb}
Hirzebruch \cite{Hirzebruch}
showed that the symmetric Hilbert modular surface $\overline{(\mathbb{H}\times \mathbb{H})/\langle PSL(2,\mathcal{O}),\tau\rangle}$
is isomorphic to the weighted projective plane
$\mathbb{P}(1:3:5)=\{(\mathfrak{A}:\mathfrak{B}:\mathfrak{C})\}$.
The point $(\mathfrak{A}:\mathfrak{B}:\mathfrak{C})=(1:0:0)$ gives the cusp $(\sqrt{-1}\infty,\sqrt{-1}\infty)$ of the modular surface.
Letting 
\begin{eqnarray}\label{XYABC}
\displaystyle X=\frac{\mathfrak{B}}{\mathfrak{A}^3},\quad\quad
Y= \displaystyle \frac{\mathfrak{C}}{\mathfrak{A}^5},
\end{eqnarray}
the pair $(X,Y)$ defines a system of  affine coordinates of $\{\mathfrak{A}\not=0\}$ of $\mathbb{P}(1:3:5)$.
\end{rem}

\begin{rem}
M\"uller \cite{Muller}
gave the  Hilbert modular forms  
$g_2$ ($s_6, s_{10},s_{15}$, resp.) of weight $2$ ($6,10,15$, resp.).
They generate the ring of Hilbert modular forms for $\mathbb{Q}(\sqrt{5})$.
\end{rem}

A $K3$ surface $X$ is a simply connected compact complex surface with $K_X=0$.
The homology group $H_2(X,\mathbb{Z})$ has the unimodular lattice structure.
Let ${\rm NS}(X)$, the N\'eron-Severi lattice of $X$, denote the sublattice in $H_2(X, \mathbb{Z})$ generated by the divisors on $X$.
The orthogonal complement ${\rm Tr}(X)$ of ${\rm NS}(X)$ in 
$H_2(X, \mathbb{Z})$ is called the transcendental lattice of $X$.  

We set the family $\mathcal{F}=\{S(\mathfrak{A}:\mathfrak{B}:\mathfrak{C})|(\mathfrak{A}:\mathfrak{B}:\mathfrak{C})\in\mathbb{P}(1:3:5)-\{(1:0:0)\}\}$ of  $K3$ surfaces with an elliptic fibration given by the affine equation
\begin{eqnarray}\label{mother}
S(\mathfrak{A}:\mathfrak{B}:\mathfrak{C}):z_0^2 =x_0^3 - 4 y_0^2(4y_0 -5\mathfrak{A}) x_0^2 + 20 \mathfrak{B} y_0^3 x_0+\mathfrak{C} y_0^4.
\end{eqnarray}

 For a generic point $(\mathfrak{A}:\mathfrak{B}:\mathfrak{C})\in\mathbb{P}(1:3:5)$, the intersection matrix of the N\'eron-Severi lattice  ${\rm NS}(S(\mathfrak{A}:\mathfrak{B}:\mathfrak{C}))$ is given by 
 $
 E_8(-1)\oplus E_8(-1) \oplus \begin{pmatrix} 2&1\\ 1&-2 \end{pmatrix} 
 $
 (see  \cite{Nagano}).
Set $\mathcal{D}=\{\xi\in\mathbb{P}^3(\mathbb{C})| 
\xi A {}^t\xi =0, \xi A {}^t\overline{\xi}>0\}$, where $A=U\oplus \begin{pmatrix} 2&1\\1&-2\end{pmatrix}$ gives the transcendental lattice of $S(\mathfrak{A}:\mathfrak{B}:\mathfrak{C})$.
Here, U is the parabolic lattice of rank $2$.
Note that  $\mathcal{D}$ is composed of $2$ connected components $\mathcal{D}_+$ and $\mathcal{D}_-$. 
We  let $(1:1:-\sqrt{-1}:0)\in \mathcal{D}_+$.
In  \cite{Nagano},
we had  the multivalued period mapping $\mathbb{P}(1:3:5)-\{(1:0:0)\}\rightarrow \mathcal{D}_+$ for $\mathcal{F}$  given by
\begin{align}\label{PERIODK3}
\Phi: (\mathfrak{A}:\mathfrak{B}:\mathfrak{C}) \mapsto \Big(\int_{\Gamma_1} \omega :\int_{\Gamma_2} \omega: \int_{\Gamma_3} \omega: \int_{\Gamma_4}\omega \Big),
\end{align}
where $\omega $ is the holomorphic $2$-form up to a constant factor and $\Gamma_1,\cdots\Gamma_4$ are   $2$-cycles on $S(\mathfrak{A}:\mathfrak{B}:\mathfrak{C})$.

\begin{rem} \label{DualRemark}
Let $\{\check{\Gamma_1},\cdots,\check{\Gamma_4}\}$ be a basis of the transcendental  lattice $A$. 
We can take $2$-cycles $\Gamma_1,\cdots,\Gamma_4$   such that they satisfy $(\Gamma_j\cdot \check{\Gamma_k})=\delta_{j,k} $ $(j,k=1,\cdots 4)$. 
These $2$-cycles $\Gamma_1,\cdots,\Gamma_4$  give the period mapping (\ref{PERIODK3}).
\end{rem}

Note that we have a biholomorphic mapping $j:\mathbb{H}\times\mathbb{H}\rightarrow  \mathcal{D}_+ $.
The multivalued mapping  $j^{-1}\circ \Phi$  on $\{\mathfrak{A}\not=0\}$
is given by
\begin{align}\label{DEVE}
(X,Y)\mapsto (z_1,z_2)=\Bigg(\displaystyle -\frac{ \displaystyle \int_{\Gamma_3}\omega + \frac{1-\sqrt{5}}{2}\int_{\Gamma_4}\omega}{ \displaystyle \int_{\Gamma_2} \omega},-\frac{ \displaystyle \int_{\Gamma_3}\omega +\frac{1+\sqrt{5}}{2}\int_{\Gamma_4} \omega}{ \displaystyle \int_{\Gamma_2}\omega}\Bigg).
\end{align}

\begin{thm}(\cite{Nagano})\label{HilbertRem}
The multivalued period mapping (\ref{DEVE}) gives a developing map of the Hilbert modular orbifold $\overline{(\mathbb{H}\times\mathbb{H})/\langle PSL(2,\mathcal{O}),\tau \rangle}$
with the branch divisor
$$
Y (-1728X^5 + 64(5X^2-Y)^2 + 720X^3Y -80XY^2 +Y^3 )= 0.
$$
The inverse of    (\ref{DEVE}) gives a pair $(X(z_1,z_2),Y(z_1,z_2))$ of symmetric Hilbert modular functions for $\mathbb{Q}(\sqrt{5})$.
\end{thm}

\begin{rem}
The icosahedral group is deeply related to the Hilbert modular functions for $\mathbb{Q}(\sqrt{5})$ (see \cite{Hirzebruch} or \cite{KobaNaru}).
Since the divisor 
\begin{eqnarray}\label{KleinIcosa}
-1728X^5 + 64(5X^2-Y)^2 + 720X^3Y -80XY^2 +Y^3=0
\end{eqnarray}
is derived from  Klein's icosahedral invariants,
this relation is called Klein's icosahedral relation.
\end{rem}

\begin{rem}
The inverse $(X(z_1,z_2),Y(z_1,z_2))$ of (\ref{DEVE})
has an explicit  expression in terms of the M\"uller's modular forms $g_2,s_6,s_{10}$   (see \cite{Nagano}).
\end{rem}

\subsection{The Kummer surface for the Humbert surface of invariant $5$}

In this subsection, we recall the properties of the Humbert surface of invariant $5$.

Let $\mathfrak{S}_2$ be the Siegel upper half plane of degree $2$.
The symplectic group $Sp(4,\mathbb{Z})$ acts on $\mathfrak{S}_2$.
The quotient space $\mathfrak{S}_2/Sp(4,\mathbb{Z})$ gives the moduli space of  principally polarized Abelian surfaces.  
Take
$\Omega=\begin{pmatrix}\sigma_1 & \sigma_2 \\ \sigma_2 & \sigma_3 \end{pmatrix}\in \mathfrak{S}_2$.
Let $L_\Omega $ be the lattice generated by the columns of the matrix $(\Omega,I_2)$.
The complex torus $Z_\Omega=\mathbb{C}/L_\Omega$ of $2$-dimension gives a principally polarized Abelian surface. 
We note that $Z_\Omega$ corresponds to the Jacobian variety  of a hyperelliptic curve of genus $2$.

Let $T $ be the involution of a $2$-dimensional complex torus $Z$ induced by $(z_1,z_2)\mapsto (-z_1,-z_2)$ on   the universal covering $\mathbb{C}^2$.
The minimal resolution ${\rm Kum}(Z)=\overline{Z/\langle id, T\rangle}$ is called the Kummer surface.
${\rm Kum}(Z)$ is a $K3$ surface.
Note that $Z$ is an Abelian surface if and  only if ${\rm Kum}(Z)$ is an algebraic $K3$ surface.

\begin{rem}
Let $\Omega\in\mathfrak{S}_2 $ and $Z_\Omega$ be the corresponding principally polarized Abelian  surface.
The Kummer surface  ${\rm Kum}(Z_\Omega)$ can be given by 
the double covering of $\mathbb{P}^2(\mathbb{C})=\{(\zeta_0:\zeta_1:\zeta_2)\}$
whose branch divisor is given by $6$ lines 
$
\zeta_2=0, \zeta_2 + 2\zeta_1 +\zeta_0=0,  \zeta_0=0
$  and
$  \zeta_2 + 2 \lambda_j \zeta_1 +\lambda_j^2 \zeta_0=0,$ $(j\in \{1,2,3\})$
with three complex parameters $\lambda_1,\lambda_2$ and $\lambda_3$.
In this paper, this Kummer surface is denoted by $K_H(\lambda_1,\lambda_2,\lambda_3)$.
\end{rem}

An element $\Omega=\begin{pmatrix} \sigma_1 & \sigma_2 \\ \sigma_2 & \sigma_3 \end{pmatrix} \in \mathfrak{S}_2$
is said to have a singular relation  with invariant $\Delta$ if there exist relatively prime integers $a,b,c,d,e\in\mathbb{Z}$ such that $a\sigma_1+b \sigma_2+c\sigma_3+d(\sigma_2^2-\sigma_1\sigma_3)+e=0$ and $\Delta=b^2-4 ac -4 de$.
Set
$
\mathcal{N}_5=\{\Omega\in\mathfrak{S}_2| \sigma  \text{ has a singular relation with invariant } \Delta\}.
$
Let $p$ be the  canonical projection $\mathfrak{S}_2 \rightarrow \mathfrak{S}_2 /Sp(4,\mathbb{Z})$.
Then, the space $\mathcal{H}_5=p(\mathcal{N}_5)$, called  the Humbert surface of invariant $5$, gives the moduli space of principally polarized Abelian surfaces $A$ such that  $\mathcal{O}\subset{\rm End }(A) $.

\begin{rem}\label{HumbertRemark}
Humbert \cite{Humbert} showed that $\Omega$ has s singular relation with $\Delta =5$ 
 if and only if 
\begin{align}\label{HumbertModular}
&\notag 4(\lambda_1^2 \lambda_3 - \lambda_2^2 +\lambda_3^2 (1-\lambda_1) + \lambda_2 ^2\lambda_3)(\lambda_1^2\lambda_2\lambda_3 -\lambda_1 \lambda_2^2 \lambda_3)\\
&=(\lambda_1^2(\lambda_2+1)\lambda_3 -\lambda_2^2 (\lambda_1 +\lambda_3) +(1-\lambda_1)\lambda_2\lambda_3^2 +\lambda_1(\lambda_2 -\lambda_3)  )^2
\end{align}
holds (see also \cite{HashimotoMurabayashi} Theorem 2.9).
This relation is called Humbert's modular equation for $\Delta=5$.
Let  $\mathcal{Q}:\mathcal{M}_{2,2}\rightarrow \mathfrak{S}_2/Sp(4,\mathbb{Z})$ be the natural projection, where $\mathcal{M}_{2,2}$ is the moduli space of genus two curves with level $2$ structure. The equation {\rm (\ref{HumbertModular})} defines a component of the inverse image $\mathcal{Q}^{-1}(\mathcal{H}_5)$.

This modular equation is studied  in detail by several researchers (for example, Hashimoto and Murabayashi \cite{HashimotoMurabayashi}).
However,
since this equation (\ref{HumbertModular}) is complicated, 
to the best of the author's knowledge, 
to study the moduli properties of the family $\{K_H(\lambda_1,\lambda_2,\lambda_3)\}$ corresponding to $\mathcal{H}_5$ is not easy.
\end{rem}

\subsection{The Shioda-Inose structure}

Let $X$ be an algebraic $K3$ surface.
Let $\omega$ be the unique holomorphic $2$-form on $X$ up to a constant factor.
If an involution $\iota: X\rightarrow X$ satisfies
$\iota^* \omega =\omega$, we call $\iota$ a symplectic  involution. 
Set $G=\langle \iota ,{\rm id} \rangle\subset {\rm Aut}(X)$.
Set $\tilde{Y}=X/G$.
Letting $Y\rightarrow \tilde{Y}$ be the minimal resolution, 
 $Y$ is a $K3$ surface.
We have  the rational quotient mapping $\chi:X \dashrightarrow Y$.

\begin{defn}
We say that a $K3$ surface $X$ admits a Shioda-Inose structure if there exists a symplectic involution 
$\iota\in{\rm Aut}(X) $ with the rational quotient mapping $\chi: X\dashrightarrow Y$ such that $Y$ is a Kummer surface and $\chi_*$ induces a Hodge isometry ${\rm Tr}(X)(2)\simeq{\rm Tr } (Y)$.
\end{defn}

\begin{thm} {\rm (Morrison \cite{Morrison})} \label{ShiodaInoseThm}
The $K3$ surface  $X$ admits a Shioda-Inose structure if and only if there is an embedding $E_8(-1)\oplus E_8(-1) \hookrightarrow {\rm NS}(X)$.
A symplectic involution $\iota$ exchanging the two copies of $E_8(-1)$ 
induces a Shioda-Inose structure.
\end{thm}

\subsection{Kummer surface $K(X,Y)$}

Due to Theorem \ref{ShiodaInoseThm}, 
the $K3$ surface $S(\mathfrak{A}:\mathfrak{B}:\mathfrak{C})$ for $(\mathfrak{A}:\mathfrak{B}:\mathfrak{C})\not=(1:0:0)$ admits a Shioda-Inose structure.
Therefore, there exists the Kummer surface $K(\mathfrak{A}:\mathfrak{B}:\mathfrak{C})$ and a symplectic involution 
$\iota $\ of $S(\mathfrak{A}:\mathfrak{B}:\mathfrak{C})$ such that the corresponding rational quotient mapping 
$\chi:S(\mathfrak{A}:\mathfrak{B}:\mathfrak{C}) \dashrightarrow K(\mathfrak{A}:\mathfrak{B}:\mathfrak{C})$ induces a Hodge isometry 
${\rm Tr}(S(\mathfrak{A}:\mathfrak{B}:\mathfrak{C}))(2) \simeq {\rm Tr}(K(\mathfrak{A}:\mathfrak{B}:\mathfrak{C}))$.

We shall obtain an  explicit defining equation of  $K(\mathfrak{A}:\mathfrak{B}:\mathfrak{C})$ by realizing the above symplectic involution $\iota$.
To find such  an involution, we need a special elliptic fibration  on $S(\mathfrak{A}:\mathfrak{B}:\mathfrak{C})$ given by the following lemma.

\begin{lem}\label{2-torsionLemma}
The defining equation of $S(\mathfrak{A}:\mathfrak{B}:\mathfrak{C})$ in {\rm (\ref{mother})} is birationally equivalent  to
\begin{eqnarray}\label{2-torsionK3}
z_1^2 = x_1 (x_1^2 +(20 \mathfrak{A} y_1^2 -20 \mathfrak{B} y_1 +\mathfrak{C})x_1 +16 y_1^5).
\end{eqnarray}
\end{lem}

\begin{proof}
Perform the birational transformation
$$
x_0=\frac{x_1}{16 y_1}, y_0=-\frac{x_1}{16 y_1^2}, z_0=\frac{x_1 z_1}{256 y_1^4}
$$
 to (\ref{mother}).
\end{proof}

\begin{figure}[h]
\center
\includegraphics[scale=1]{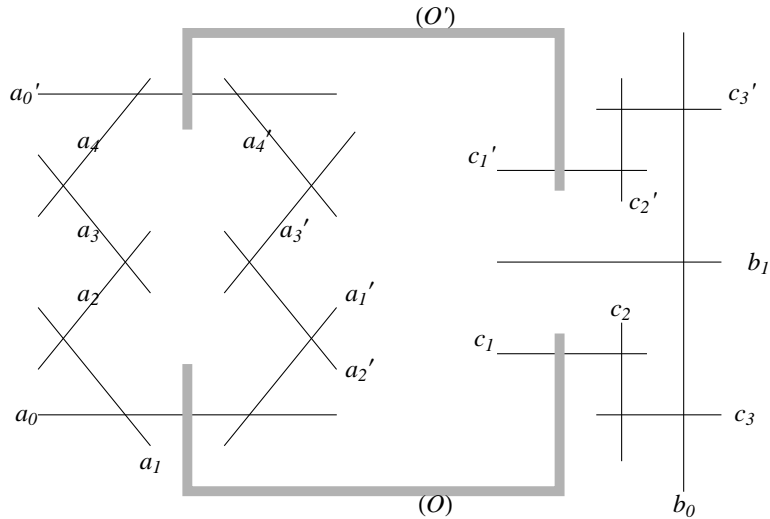}
\caption{The singular fibres given by (\ref{2-torsionK3}).}
\label{singularFig}
\end{figure}

The mapping $\pi_1:S(\mathfrak{A}:\mathfrak{B}:\mathfrak{C})\rightarrow \mathbb{P}^1(\mathbb{C})$  given by $(x_1,y_1,z_1)\mapsto y_1$ 
defines an elliptic fibration.
The fibre $\pi_1^{-1}(0)$ ($\pi_1^{-1}(\infty)$, resp.) is a singular fibre of $\pi_1$ of type $I_{10}$ ($III^*$, resp.). 
We set $\pi_1^{-1}(0) = a_0+a_1+\cdots+a_4+a_0'+a_1'+\cdots+a_4' $ and $\pi_1^{-1}(\infty)=b_0+b_1+c_1+c_2+c_3+c_1'+c_2'+c_3'$. 
Let $O$ be the zero of the Mordell-Weil group.
Let $O'$ be the section of $\pi_1$ given by $(x_1,y_1,z_1)=(0,y_1,0)$.
Note that $2O'=O$ (see Figure 1).

We have the involution $\iota$ of $S(\mathfrak{A}:\mathfrak{B}:\mathfrak{C})$ given  by
$$
 (x_1,y_1,z_1)\mapsto  (\frac{16 y_1^5}{x_1},y_1,\frac{-16 y_1^5 z_1}{x_1^2} ).
$$
This is a symplectic involution.
Note that $\iota$ is a van Geemen-Sarti involution 
for elliptic  surfaces  {\rm (see \cite{GeemenSarti})}.
Let $G=\langle id, \iota \rangle$.
Set
\begin{align}\label{u1v1}
\displaystyle u_1=x_1+\frac{16 y_1^5}{x_1} ,\quad\quad
\displaystyle v_1=\frac{x_1 ^2- 16 y_1^5}{z_1}.
\end{align}
They are $G$-invariants.
We can see that
$
(x_1,y_1,z_1)\mapsto (u_1,y_1,v_1)
$
defines a $2$ to $1$ mapping.

\begin{thm}
The defining equation of the Kummer surface $K(\mathfrak{A}:\mathfrak{B}:\mathfrak{C})$ is given by
\begin{eqnarray}\label{KummerABC}
 v^2=(u^2-2y^5)(u-(5\mathfrak{A}y^2-10\mathfrak{B}y+\mathfrak{C})).
\end{eqnarray}
For generic $(\mathfrak{A}:\mathfrak{B}:\mathfrak{C})\in \mathbb{P}(1:3:5)$, 
the intersection matrix of the transcendental lattice ${\rm Tr}(K(\mathfrak{A}:\mathfrak{B}:\mathfrak{C}))$
is given by
$$
A(2)= \begin{pmatrix} 0 & 2 \\ 2 &0 \end{pmatrix} \oplus \begin{pmatrix} 4 & 2 \\ 2 & -4 \end{pmatrix}. 
$$
\end{thm}

\begin{proof}
By direct observations, we can check that
 $\iota$ interchanges $2$ copies of  $E_8 (-1)$ in ${\rm NS}(S(\mathfrak{A}:\mathfrak{B}:\mathfrak{C}))$ (see Figure 2).
Therefore, due to Theorem \ref{ShiodaInoseThm}, the involution $\iota$ gives the Shioda-Inose structure on $S(\mathfrak{A}:\mathfrak{B}:\mathfrak{C})$.
\begin{figure}[h]
\label{E8fig}
\center
\includegraphics[scale=1]{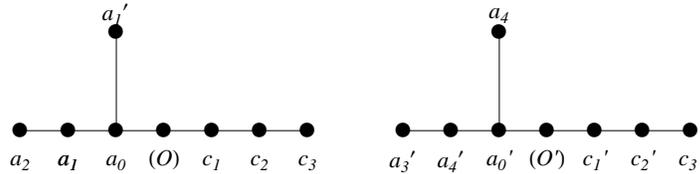}
\caption{$E_8(-1)$ lattices in ${\rm NS}(S(\mathfrak{A}:\mathfrak{B}:\mathfrak{C})).$}
\end{figure}

From (\ref{2-torsionK3}), (\ref{u1v1}),  and the birational transformation 
$$
u_1=-u,\quad v_1=\frac{\sqrt{-1}v}{u-(5\mathfrak{A}y^2-10 \mathfrak{B}y +\mathfrak{C})},\quad y_1=\frac{y}{2},
$$
we can check that
 the defining equation of $S((\mathfrak{A}:\mathfrak{B}:\mathfrak{C})/G$
is given by (\ref{KummerABC}).

We have the intersection matrix of ${\rm Tr}(K(\mathfrak{A}:\mathfrak{B}:\mathfrak{C}))$ since $\iota$ gives the Shioda-Inose structure.
\end{proof}

We have the family $\tilde{\mathcal{K}}=\{K(\mathfrak{A}:\mathfrak{B}:\mathfrak{C})\}$ of Kummer surfaces. 
The projection $(y,u,v)\mapsto (y,u)$   defines   the double covering $\mathcal{P}:K(\mathfrak{A}:\mathfrak{B}:\mathfrak{C})\rightarrow \mathbb{P}(1:1:2)=\{(\zeta_0:\zeta_1:\zeta_2)\}$, where $y=\frac{\zeta_1}{\zeta_0}$ and $u=\frac{\zeta_2}{\zeta_0^2}$ on $\{\zeta_0\not=0\}$.
Its branch divisor is given by $\tilde{P}\cup \tilde{Q}$, where
\begin{align}\label{PandQ}
\begin{cases}
& \tilde{P} \cap \{\zeta_0\not=0\}=\{(y,u) |  u= 5 \mathfrak{A} y^2 -10 \mathfrak{B} y + \mathfrak{C} \},\\
& \tilde{Q}\cap\{\zeta_0\not=0\}=\{ (y,u)|  u^2 =2y^5\}.
\end{cases}
\end{align}

\begin{rem}
The equation  (\ref{KummerABC}) gives an expression of the Kummer surface ${\rm Kum}(Z_\Omega)$  for $\Omega\in\mathcal{H}_5$.
This is different from the expression of $K_H(\lambda_1,\lambda_2,\lambda_3)$ in Remark \ref{HumbertRemark}.
Our expression has some advantages. 
For example,
 our parameter space has a simple compactification by adding the point $(\mathfrak{A}:\mathfrak{B}:\mathfrak{C})=(1:0:0)$.
 This point is equal to the cusp of the Hilbert modular surface $\overline{(\mathbb{H}\times \mathbb{H})/\langle PSL(2,\mathcal{O}),\tau \rangle}$ (see Remark \ref{RemarkHilb}).
\end{rem}

Let $\omega_K$ be the unique holomorphic $2$-form on $K(\mathfrak{A}:\mathfrak{B}:\mathfrak{C})$ up to a constant factor.
Set 
$
\chi_* (\Gamma_j) =\Delta_j
$
 for $j\in\{1,2,3,4\}$.
We have the period mapping for $\mathcal{K}$
given by
\begin{align}\label{KummerPeriod}
\Phi_K: (\mathfrak{A}:\mathfrak{B}:\mathfrak{C}) \mapsto \Big(\int_{\Delta_1} \omega_K :\int_{\Delta_2} \omega_K: \int_{\Delta_3} \omega_K: \int_{\Delta_4}\omega_K \Big)\in\mathcal{D}.
\end{align}

Since $\chi^*(\omega_K)=\omega$ and $\chi_* (\Gamma_j) =\Delta_j$,
we have clearly the following proposition.

\begin{prop}\label{equalLem}
It holds that
\begin{align*}
\Big( \int_{\Gamma_1}\omega:\cdots:\int_{\Gamma_4}\omega \Big) =\Big( \int_{\Delta_1}\omega_K:\cdots:\int_{\Delta_4}\omega_K \Big).
\end{align*}
\end{prop}

\begin{rem}\label{StoK}
According to Theorem \ref{HilbertRem} and the above proposition,
the inverse of $j^{-1}\circ \Phi_K$ gives the pair $(X,Y)$ of Hilbert modular functions for $\mathbb{Q}(\sqrt{5})$ via the period mapping $\Phi_K$.
\end{rem}

We have the projection $\pi:K(\mathfrak{A}:\mathfrak{B}:\mathfrak{C})\rightarrow \mathbb{P}^1(\mathbb{C})$ given by $(u,y,v) \mapsto y$.
We have the singular fibre $\pi^{-1} (0)$ ($\pi^{-1} (\infty)$, resp.) of the elliptic surface $(K((\mathfrak{A}:\mathfrak{B}:\mathfrak{C}),\pi,\mathbb{P}^1(\mathbb{C}))$ of type $I_5$ ($III^*$, resp.) and  other five singular fibres $\pi^{-1}(s_1),\cdots, \pi^{-1}(s_5)$ of type $I_2$.

\begin{prop}\label{NSKummerProp}
The vector space
${\rm NS}(K(\mathfrak{A}:\mathfrak{B}:\mathfrak{C}))\otimes_\mathbb{Z} \mathbb{Q}$ is generated by the components of singular fibres,  the section  $O$ given by the zero of the Mordel-Weil group and a general fibre $F $ of $\pi$.
\end{prop}

\begin{proof}
${\rm NS}(K(\mathfrak{A}:\mathfrak{B}:\mathfrak{C}))\otimes_\mathbb{Z} \mathbb{Q}$ is an $18$-dimensional vector space over $\mathbb{Q}$.
Set $\displaystyle\pi^{-1}(y)=\bigcup_{j=0,\cdots,r(y)} \Theta_{y,j}$, where $\Theta_{y,j}$ is a connected component and $\Theta_{y,0}\cap O\not=\phi$. 
By calculating the intersection numbers,
we can check that the $18$ divisors $\Theta_{0,1},\cdots,\Theta_{0,4},$ $\Theta_{s_1,1},\cdots,\Theta_{s_5,1},$ $\Theta_{\infty,1},\cdots,\Theta_{\infty,7}$, $O$ and $F$
generate a sublattice of ${\rm NS}(K(\mathfrak{A}:\mathfrak{B}:\mathfrak{C}))$  of rank $18$. 
Hence the claim follows.
\end{proof}

Considering the relation (\ref{XYABC}), 
we have the defining equation (\ref{K(X,Y)}) with $(X,Y)$ parameters.

\section{The double integrals of an algebraic function on chambers surrounded by a parabola and a quintic curve}

In this section, we  obtain an extension of the classical elliptic integrals.
We shall study a single-valued branch  $U_0\rightarrow\mathcal{D}_+$ of the multivalued period mapping $\Phi_K$ explicitly
where $U_0$ is an open set given by Figure 3 in $\mathbb{R}^2$.
By the analytic continuation of this single-valued branch, we have the multivalued period mapping $\Phi_K$ in (\ref{KummerPeriod}). 
The arrangement of $P$ in (\ref{P}) and $Q$ in (\ref{Q}) determines the chambers $R_1,R_2,R_3$ and $R_4$ in Figure 9.
Theorem  \ref{DoubleIntegralThm} gives an extension of the classical elliptic integrals to the Hilbert modular case for $\mathbb{Q}(\sqrt{5})$.

\subsection{The elliptic curve $E(y)$}

For $y>0$, set
$
\alpha(y) =y^2 \sqrt{2y},
\beta (y) = -y^2 \sqrt{2y}$ and
$p(y)=5y^2 -10X y + Y$
where $\sqrt{y}>0$.
Note that these $\alpha(y),\beta(y)$ and $p(y)$ are real valued analytic functions for $y\in\mathbb{R}_+$. 
Set
\begin{eqnarray}\label{E(y)}
E(y): v^2=(u-\alpha(y))(u-\beta(y))(u-p(y)),
\end{eqnarray}
for $y\in\mathbb{R}_+$.
Of course, $E(y)$ gives the  fibre for $y\in\mathbb{R}_+$ of the elliptic  surface $(K(X,Y),\pi,\mathbb{P}^1(\mathbb{C}))$.
The discriminant of the right hand side  of (\ref{E(y)}) for $u$ has five roots in $y$-plane.

Let $U_0$ be the domain in $\mathbb{R}^2=\{(X,Y)\}$ described in  Figure 3.
The curve in Figure 3 is Klein's icosahedral relation in (\ref{KleinIcosa}).
 \begin{figure}[h]
\center
\includegraphics[scale=.4,clip]{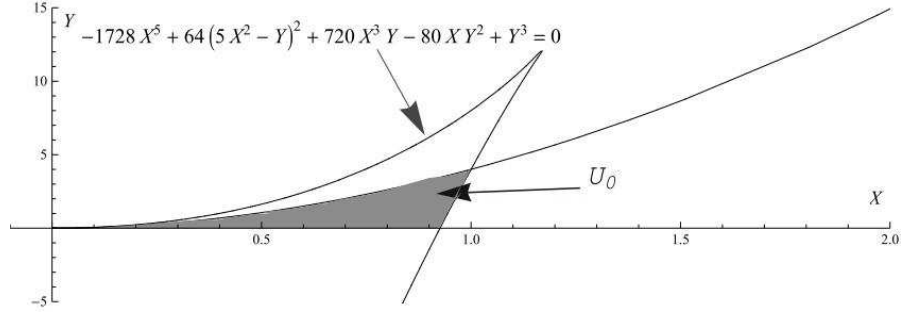}
\caption{The domain $U_0$ in $(X,Y)$-space $\mathbb{R}^2$.}
\end{figure}
If $\displaystyle (X,Y)\in U_0$,  
the five roots of the discriminant of the right hand side  of (\ref{E(y)}) for $u$ are in $\mathbb{R}_+(\subset y{\rm -space})$.
So, we  let $s_1=s_1(X,Y),s_2=s_2(X,Y),s_3=s_3(X,Y),s_4=s_4(X,Y)$ and $s_5=s_5(X,Y)$ be these five roots such that $0<s_1<s_2<s_3<s_4<s_5$.

For $(X,Y)\in U_0 $ and $s_{j-1}<y<s_j$ $(j=0,\cdots,6)$, we denote the right hand side of $E(y)$ by
$
(u-w_1 (y))(u-w_2 (y))(u-w_3 (y)),
$
where $w_1(y) <w_2(y) <w_3(y)$
(see Table 2 and Figure 4). 
\begin{table}[h]\label{w-Table}
\center
{\small
\begin{tabular}{ccccccc}
\toprule
&{\footnotesize $0<y<s_1$}&{\footnotesize$s_1<y<s_2$}&{\footnotesize $s_2<y<s_3$} &{\footnotesize$s_3<y<s_4$ }&{\footnotesize$s_4<y<s_5$ }&  {\footnotesize$s_5<y$} \\
\midrule 
$w_1(y)$ & $\beta(y)$& $\beta(y)$& $p(y)$ & $\beta(y)$& $\beta(y)$ & $\beta(y)$\\
$w_2(y)$ & $\alpha(y)$ & $p(y)$ & $ \beta(y)$& $p(y)$ & $\alpha(y)$ & $p(y)$\\
$w_3(y)$ & $p(y)$ & $\alpha(y)$&$\alpha(y)$ &$\alpha(y)$ &$p(y)$ &$\alpha(y)$\\
\bottomrule
\end{tabular}}
\caption{The correspondence between $\{w_1,w_2,w_3\}$ and $\{\alpha,\beta,p\}.$} 
\end{table}

\begin{figure}\label{abp}
\center
\includegraphics[scale=0.62]{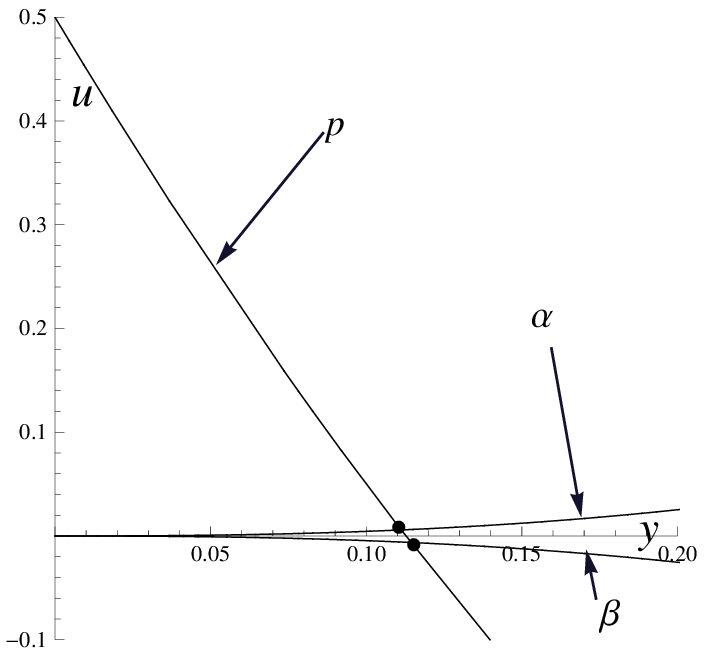}
\quad
\includegraphics[scale=0.6]{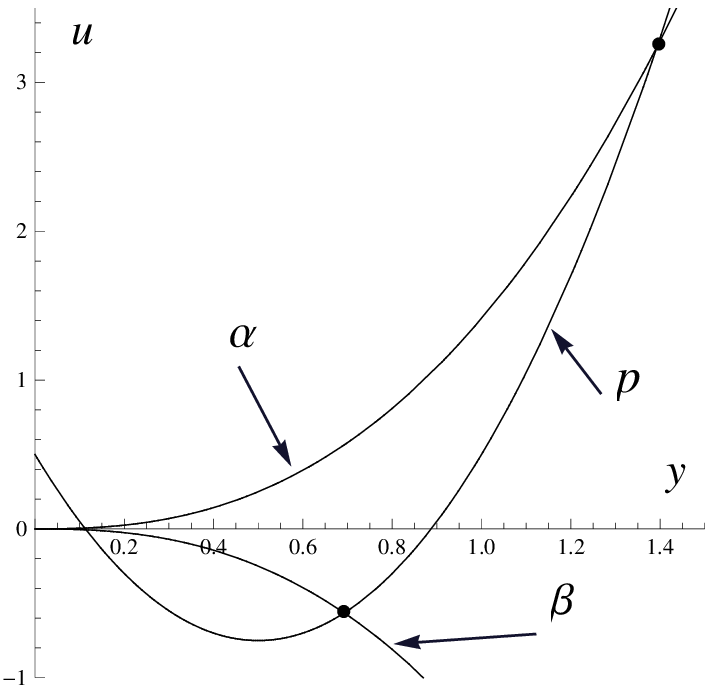}
\quad
\includegraphics[scale=0.53]{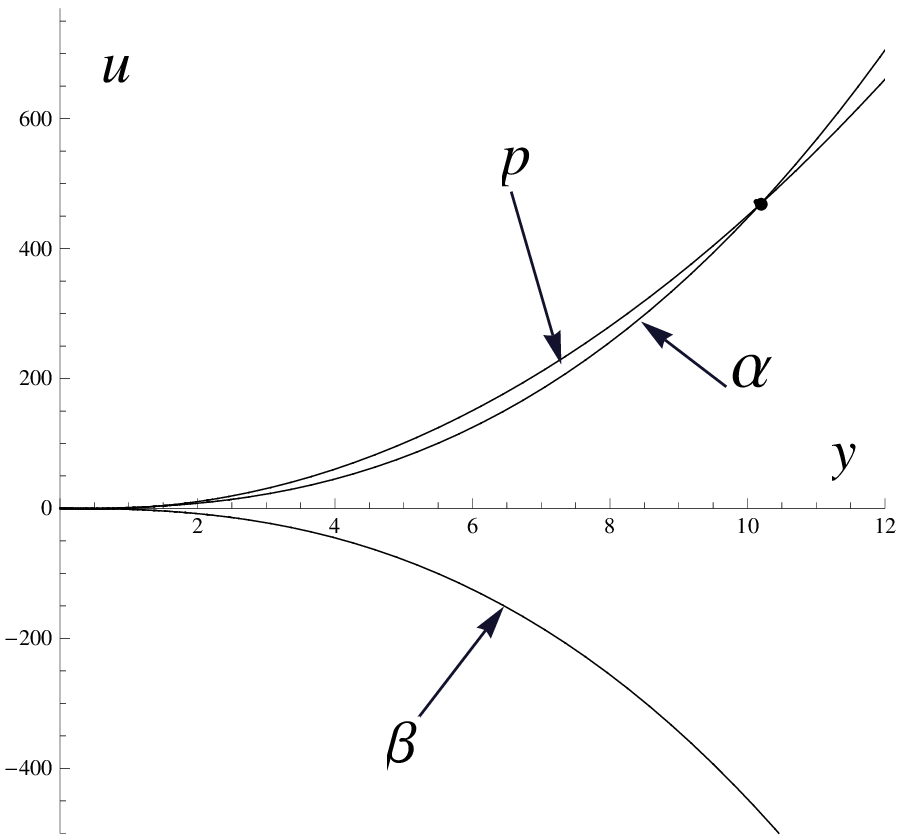}\\
{\small
$0<y<0.2$   \hspace{2cm} $0<y<1.5$ \hspace{2cm} $0<y<12$
}
\caption{The graph of $u=\alpha(y),\beta(y),p(y)$
}
\end{figure}

Since the points $\alpha(y),\beta(y)$ and $p(y)$ are real-valued for $y\in\mathbb{R}_+$, 
the function 
$$
F(y,u_+)=\sqrt{(u_+-\alpha(y))(u_+-\beta(y))(u_+-p(y))}
$$
is single-valued   on
 $\{(y,u_+)|y\in\mathbb{R}_+, {\rm Im}(u_+)>0\}$.
Hence, 
\begin{align}\label{alg.F}
F(y,u)=\lim_{t \rightarrow 0} F(y,u+\sqrt{-1} t) \in\mathbb{R}
\end{align}
is  single-valued 
for $s_{j-1}<y<s_j$ and $u\not\in\{ \alpha(y),\beta(y),p(y),\infty\}$ as Table 3. 
\begin{table}[h]
\center
{\small
\begin{tabular}{ccccc}
\toprule
  &$-\infty<u<w_1$ &  $w_1<u<w_2$ & $w_2<u<w_3$ & $ w_3<u<\infty$ \\  
  \midrule
$F(u,y)$ & $-\sqrt{-1}\mathbb{R}_+$& $-\mathbb{R}_+$& $\sqrt{-1}\mathbb{R}_+$ &$\mathbb{R}_+$ \\
\bottomrule
\end{tabular}
}
\caption{The values of $F(u,y) $.}
\end{table}

Take a base point $b\in (s_2,s_3)(\subset \mathbb{R})$.
We can take the basis $\{\gamma_1 , \gamma_2\}$ of the homology group $H_1(\pi^{-1}(b),\mathbb{Z})$ such that $(\gamma_1\cdot\gamma_2)=1$  and 
\begin{eqnarray*}
{\tiny
\int_{\gamma_1} \omega = 2\int _{\beta(b)}^{p(b)}\frac{du}{\sqrt{F(b,u)}},\quad
\int_{\gamma_2} \omega =2 \int _{\alpha(b)}^{\beta(b)}\frac{du}{\sqrt{F(b,u)}}.
}
\end{eqnarray*}

For $j\in\{0,1,2\}$ ($\in\{3,4,5\}$, resp.), 
we put $l_j= \{(s_j,-\sqrt{-1} t ) | t\geq 0\}$
 ($=\{(s_j,\sqrt{-1}t ) | t\geq0 \}$, resp.).
We call $l_j$ the cut line for $s_j$.
For $y\in\mathbb{C}-\{l_0,\cdots,l_5\}$, 
take an arc $\alpha_y$ which does not touch the cut lines $l_j$ ($j\in\{0,\cdots,5\}$) with the start (end, resp.) point $b$ ($y$, resp.).
 Let $u\mapsto a_y(u)$ $(0\leq u\leq 1)$ be the parametric representation of $\alpha_y$.
Take a $1$-cycle $\gamma$ on $E(b)$.
For $\gamma \in H_1(\pi^{-1}(b),\mathbb{Z})$,
we choose the $1$-cycle $\gamma_{\alpha_y}(u)$ on $\pi^{-1}(a_y(u))$ which depends continuously on $u$ with $\gamma_{\alpha_y}(0)=\gamma$.
If $\alpha'_y$ is   homotopic to $\alpha_y$ in $\mathbb{C}-\{l_0\cup\cdots\cup l_5\}$, we have $\gamma_{\alpha_y}(1)=\gamma_{\alpha'_y}(1)$.
So,  we have a well-defined correspondence 
$
\mathbb{C}-\{l_0\cup\cdots\cup l_5\}\ni y\mapsto\gamma_{\alpha_y}(1)\in H_1(\pi^{-1}(y),\mathbb{Z}).
$
Then, we put
\begin{eqnarray}\label{gamma(y)}
\gamma=\gamma_{\alpha_y}(1) \in H_1(\pi^{-1}(y),\mathbb{Z})    \quad\quad (y\in\mathbb{C}-\{y_0,\cdots,y_5\}).
\end{eqnarray}

Next,  let $r_j$ $(j=0,1, \cdots, 5)$ be a closed arc on $\mathbb{C}-\{0,s_1 , \cdots, s_5\}$, starting at $b$, goes around $s_j $ with the positive orientation and ending at $b$. 
We assume that $r_j$ does not touch the cut line $l_k$ if $j\not= k$.
Let $t\mapsto u_j(t)$ $(0\leq t \leq 1)$ be the parametric representation of $r_j$.
For instance, we can take an arc $r_1$ as in Figure 5.
We choose $1$-cycles $\gamma_1(t)$ and $\gamma_2 (t)$ on $\pi^{-1}(u_j(t))$ which depend continuously on $t$ such that $\gamma_1(0)=\gamma_1$ and $\gamma_2(0)=\gamma_2$.
So, we have 
\begin{align*}
\begin{pmatrix} \gamma_1(1) \\  \gamma_2(1) \end{pmatrix}
=\begin{pmatrix} a_j & b_j \\ c_j & d_j \end{pmatrix}
 \begin{pmatrix} \gamma_1 \\  \gamma_2 \end{pmatrix},
\end{align*} 
where $a_j,b_j,c_j,d_j \in \mathbb{Z}$ and $a_j d_j-b_j c_j =1$.
The correspondence $\displaystyle r_j\mapsto M_j=\begin{pmatrix}a_j & b_j \\ c_j & d_j \end{pmatrix}$ gives a representation of the fundamental group $\pi_1(\mathbb{C} -\{0,s_1,\cdots,s_5\})$.
We call the matrix $M_j$ the monodromy matrix for $r_j$.

\begin{figure}[h]\label{cutFigure}
\center
\includegraphics{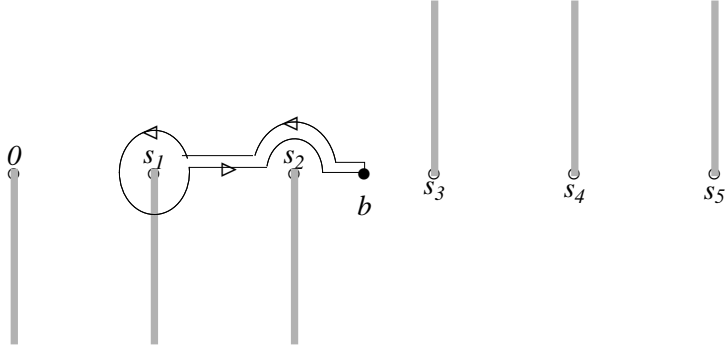}
\caption{The points $0,s_1,\cdots,s_6$, the cut lines and an arc $r_1$ going around $s_1$.}
\end{figure}

\begin{rem}
If an arc $r$ in the base space of an elliptic fibration goes around  a  singular fibre  with the positive orientation,
the monodromy matrix $M_r$ is obtained by 
K. Kodaira (\cite{Kodaira}, Theorem 9.1).
For example, if the  singular fibre  is of type $I_b$ $(b>0)$ or $III^*$, the monodromy matrix $M_r$ is given by $B^{-1} M_r^0 B$, where $M_r^0$ is given by Table 4 and $B\in GL(2,\mathbb{Z})$.
\begin{table}[h]
\center
\begin{tabular}{lc}
\toprule
Singular Fibre  &  Matrix $M_r^0$ \\  
  \midrule
$I_b$ & $\begin{pmatrix} 1& 0\\b&1 \end{pmatrix} $\\
$III^*$ &$\begin{pmatrix} 0& 1\\-1&0 \end{pmatrix} $\\
\bottomrule
\end{tabular}
\caption{The  matrices $M_r^0$ for the singular fibres  of type $I_b$ and $III^*$.}
\end{table} 
\end{rem}

\begin{lem}\label{LocalMonodoromyLemma}
The monodromy matrices $M_j$ for $\{\gamma_1,\gamma_2\}$ are given by {\rm Table 5}.
\end{lem}

\begin{table}[h]
\center
\begin{tabular}{cccc}
\toprule
 & Type of Singular Fibre  &  Monodromy Matrix for $\gamma_1 , \gamma_2$ \\  
  \midrule
$y_1$ & $I_2$& $M_1=\begin{pmatrix} 3& -2\\2&-1 \end{pmatrix}$ \\
$y_2$ & $I_2$ &$M_2=\begin{pmatrix} 1& 0\\ 2&1 \end{pmatrix}$ \\
$y_3$ & $I_2$ &$M_3=\begin{pmatrix} 1& 0\\2&1 \end{pmatrix}$ \\
$y_4$ & $I_2$ &$M_4=\begin{pmatrix} 3& -2\\2&-1 \end{pmatrix}$ \\
$y_5$ & $I_2$ &$M_5=\begin{pmatrix} 3& -2\\2&-1 \end{pmatrix}$ \\
$0$ & $I_5$ &$M_0=\begin{pmatrix} 1& -5\\ 0&1 \end{pmatrix}$ \\
$\infty$ & $III^*$ &$M_\infty=\begin{pmatrix} 3& 5\\ -2&-3 \end{pmatrix}$ \\
\bottomrule
\end{tabular}
\caption{The monodromy matrices $M_j$ $(j=0,1,\cdots,5,\infty)$.}
\end{table}

\begin{proof}
Let us determine matrix $M_2$ around $s_2$.
The fibre $\pi^{-1}(s_2)$ is a singular fibre  of type $I_2$. So, the monodromy matrix $M_2$ is in the form
\begin{align*}
M_2= B^{-1} \begin{pmatrix} 1 &0\\2 &1 \end{pmatrix} B,
\end{align*}
where $B\in GL(2,\mathbb{Z})$.
Observe that $p(y)=w^{(3)}_1(y)$  converges to $\beta(y)=w_2^{(3)}(y)$ when $y\rightarrow y_2+0$.  So, the matrix $M_2$  fixes the $1$-cycle $\gamma_1=\gamma_1^{(3)}$.
Hence, we have $B=I_2$ and $\displaystyle M_2=\begin{pmatrix} 1 &0\\2 &1 \end{pmatrix}.$
By the same argument, we  obtain Table 5. 
\end{proof}

\subsection{The transcendental lattice $\langle D_1,\cdots,D_4\rangle$}

From Table 5, we have the following relations:

\begin{align}
\begin{cases}
& M_1 M_2 M_4 M_3 =\begin{pmatrix}  1 & 4 \\ 0 & 1\end{pmatrix}, \quad 
M_1 M_2 M_5 M_3  =\begin{pmatrix}  1 & 4 \\ 0 & 1\end{pmatrix},\\
&M_2^{-1} M_3 =\begin{pmatrix} 1 & 0 \\ 0&1  \end{pmatrix}, \quad
M_0 M_1M_2 M_0^{-1} M_3^{-1} =\begin{pmatrix} 3 & -2 \\ 2&-1  \end{pmatrix}.   \\
\end{cases}
\end{align}

The transformation given by the matrix $M_1M_2M_4M_3$ fixes the $1$-cycle $\gamma_2$.
Let $\rho_1$ be a closed curve in $y-$plane  starting from the base point $b$  
and goes around $s_1,s_2,s_4$ and $s_3$ successively. 
Let $t\mapsto s(t)$ be a parametric representation of $\rho_1$.
For $0\leq t\leq 1$, we put the $1$-cycle $\gamma^{(1)}(t)$
 on the elliptic curve $\pi^{-1}(s(t))$.
 The $1$-cycle  $\gamma^{(1)}(t)$ depends continuously on $t$ and
   $\gamma^{(1)}(0)=\gamma^{(1)}(1)=\gamma_2$ on $\pi^{-1}(b)=\pi^{-1}(s(0))=\pi^{-1}(s(1))$.
Then, the set
$$
C_1=\displaystyle \bigcup_{0\leq t \leq 1}  \gamma^{(1)}(t)
$$
defines a $2$-cycle on the  surface $K(X,Y)$. 
Similarly, we have the $2$-cycles $C_2,C_3$ in Figure 6 and $C_4$ in Figure 7.

\begin{figure}\label{C1C2C3Fig}
\center
\includegraphics{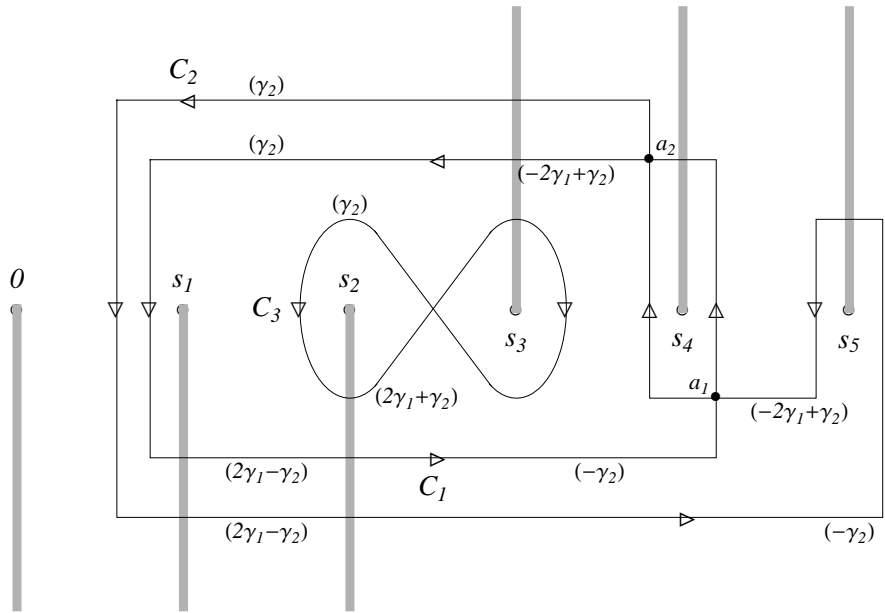}
\caption{$2$-cycles $C_1,C_2,C_3$.}
\end{figure}

\begin{figure}
\center
\includegraphics{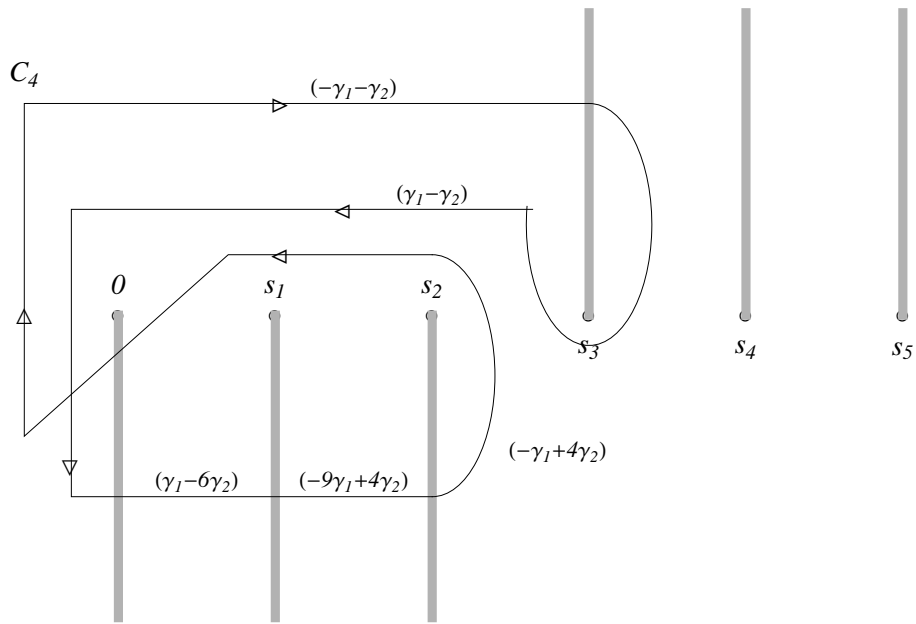}
\caption{$2$- cycle $C_4.$}
\end{figure}

\begin{lem}\label{LemC1C2C3C4}
The intersection matrix for $\{C_1,C_2,C_3,C_4\}$ is given by
\begin{align*}
 ( (C_j\cdot C_k))_{j,k=1,\cdots,4}=\begin{pmatrix} 0&2&0&0 \\ 2&0&0&0\\ 0&0&-4&-6 \\ 0&0&-6&-4 \end{pmatrix}.
\end{align*}
\end{lem}

\begin{proof}
Let $\rho_j$ be the base arc of $C_j$.
For $y\in \rho_{j} $, let $\gamma^{(j)} (y)=C_j \cap \pi^{-1}(y)$.
Suppose the base arcs $\rho_j$ and $\rho_k$ intersect at   $s$ points $y_1,\cdots,y_s$ in $y$-plane. 
Then, the intersection number $(C_j\cdot C_k)$ is given by the following formula:
\begin{eqnarray}\label{intersectionFormula}
\displaystyle (C_j\cdot C_k)=\sum _{l=1}^s (-1) (\rho_{j}  \cdot \rho_{k} )_{y_l} (\gamma^{(j)}(y_l) \cdot \gamma^{(k)} (y_l)),
\end{eqnarray}
where $(\rho_j  \cdot \rho_k )_{y_l} $ is the intersection number of the base arcs $\rho_j$ and $\rho_k$ at the point $y_j$ and $(\gamma^{(j)}(y_l) \cdot \gamma^{(k)} (y_l))$ is the intersection number of $1$-cycles on the elliptic curve $\pi^{-1}(y_j)$.
See Figure 6.
The base arc $\rho_1$ and $\rho_2$ intersect at $2$ points $a_1$ and $a_2$.
We have  $(\rho_1\cdot\rho_2)_{a_1}=+1$  and $(\rho_1\cdot\rho_2)_{a_2}=-1$.
Then, from (\ref{intersectionFormula}), we have
$$
(C_1\cdot C_2)=(-1)(+1)(-\gamma_2\cdot -2\gamma_1+\gamma_2)+(-1)(-1)(-2\gamma_1+\gamma_2\cdot-2\gamma_1+\gamma_2)=(-1)(-2)+0=2.
$$
By the same argument, the claim follows.
\end{proof}

Due to the above lemma, the following corollary is obvious.

\begin{cor}\label{D1D2D3D4Cor}
Put
\begin{align}
D_1=C_1,
\quad
D_2=C_2,\quad
D_3=C_4-C_3,\quad
D_4=C_4.
\end{align}
Then, the intersection matrix for $\{D_1,\cdots,D_4\}$ is given by 
\begin{align}
 ( (D_j\cdot D_k))_{j,k=1,\cdots,4}=\begin{pmatrix} 0&2&0&0 \\ 2&0&0&0\\ 0&0&4&2 \\ 0&0&2&-4 \end{pmatrix}.
\end{align}
\end{cor}

\begin{prop}\label{TransProp}
The system
$\{D_1,D_2,D_3,D_4\} $ gives a basis of the transcendental lattice of $K(X,Y)$ 
with the intersection matrix $A(2)$.
\end{prop}

\begin{proof}
By the above construction, $2$-cycle $D_j$ $(j=1,\cdots,4) $ does not touch the singular fibres of $(K(X,Y),\pi,\mathbb{P}^1(\mathbb{C}))$. 
So, from Theorem \ref{HilbertRem} and Proposition \ref{NSKummerProp}, 
the system $\{D_1,\cdots ,D_4\}$  gives a basis of ${\rm Tr} (K(X,Y))$.
\end{proof}

\subsection{The $2$ cycles $L_1,\cdots,L_6$}

Next, we define the $2$-cycles $L_1,\cdots,L_6$ on $K(X,Y)$.
Let $\varrho_j$ $(j=1,\cdots,6)$ be an arc in $y$-plane with a parametric representation $t\mapsto q_j(t)$ $(0\leq t \leq 1)$ whose start point and end point is given by Table 6.
Remark that  we take them such that $\varrho_j$ does not touch the cut lines $l_k$ $(k\in \{0,\cdots,5\})$ if $0<t<1$.
Hence,  we  can put a $1$-cycle $\delta^{(j)} (q_j(t))$  on $\pi^{-1}(q_j(t))$ as Table 6   with the manner in (\ref{gamma(y)}).  
Then, we can see that
$
L_j=\displaystyle \bigcup _{0\leq t\leq 1} \delta^{(j)} (q_j(t))
$
gives a $2$-cycle on $K(X,Y)$ (see Figure 8).

\begin{table}[h]\center
\begin{tabular}{ccccccc}
\toprule
  & $L_1$ &  $L_2$ & $L_3$ & $L_4$& $L_5$ & $L_6$ \\  
  \midrule
start point of $\varrho_j$ & $s_5$& $s_4$& $s_3$  &$s_2$ &$s_1$ &$0$\\
end point of $\varrho_j$ & $\infty$ &$s_5$& $ \infty$ & $s_3$ &$s_4$ & $\infty$ \\
$1$-cycle $\delta^{(j)}$ & $\gamma_1-\gamma_2$& $\gamma_1-\gamma_2$ &$\gamma_1$ &$\gamma_1$ & $\gamma_1-\gamma_2$ & $\gamma_2$\\
\bottomrule
\end{tabular}
\caption{The  arc $\varrho_j$  and $1$-cycles for $2$-cycles $L_j$ $(j=1,\cdots,6).$}
\end{table}

As we proved Lemma \ref{LemC1C2C3C4}, we can prove the following lemma and corollary.

\begin{lem}
\begin{align}
((L_j\cdot C_k))_{1\leq j\leq 6,1\leq k\leq 4}=\begin{pmatrix} 0&1&0&0 \\ 1&-1&0&0 \\ 1&1&1&-1 \\ 0&0&-2&-3 \\ 0&1&-2&-3 \\ 0&0&2&0 \end{pmatrix}.
\end{align}
\end{lem}

\begin{cor}\label{LDCor}
\begin{align}
((L_j\cdot D_k))_{1\leq j\leq 6,1\leq k\leq 4}=\begin{pmatrix} 0&1&0&0 \\ 1&-1&0&0 \\ 1&1&-2&-1 \\ 0&0&-1&-3 \\ 0&1&-1&-3 \\ 0&0&-2&0 \end{pmatrix}.
\end{align}
\end{cor}

\begin{figure}[h]
\center
\includegraphics[width=10cm,height=4cm]{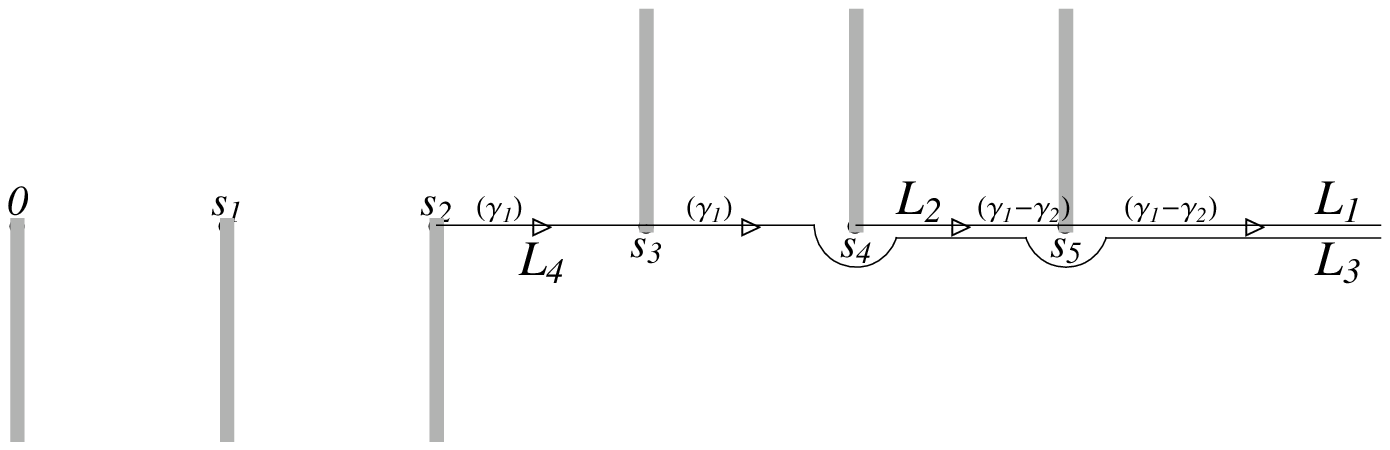}
\\
\includegraphics[width=10cm,height=5cm]{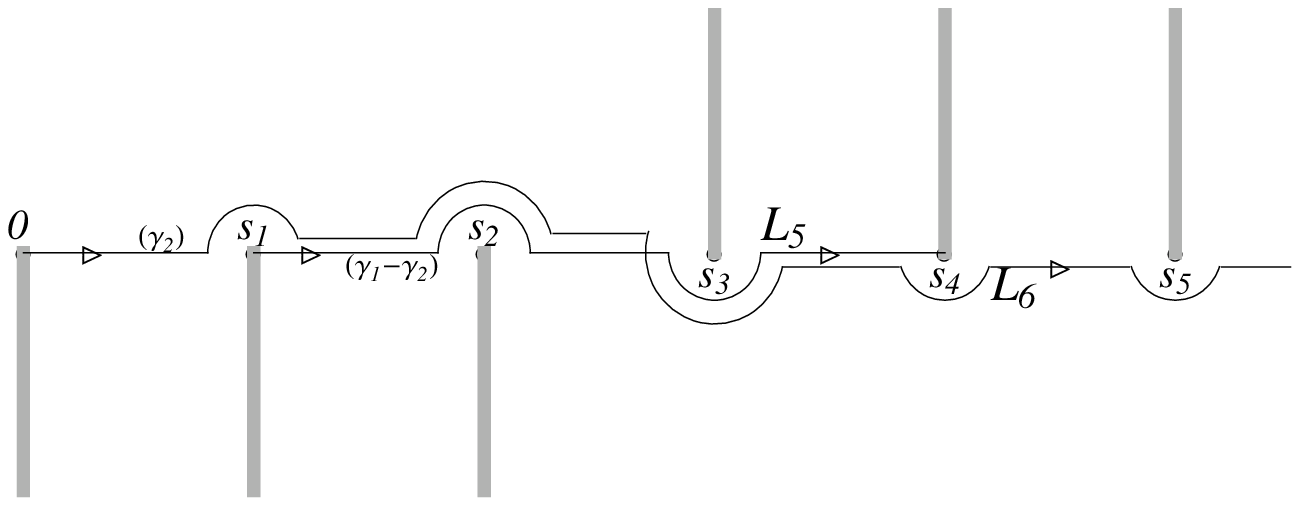}
\caption{$2$-cycles  $L_1,L_2,L_3,L_4, L_5$ and $L_6$.}
\end{figure}

\begin{prop}\label{DeltaThm}
A branch of the period mapping $\Phi_K$ in {\rm (\ref{KummerPeriod})} on $U_0$ has the following expression:
\begin{align}\label{intlemma}
\begin{cases}
\vspace{2mm}
&\displaystyle \int _{\Delta_1} \omega_K=\int_{ L_1+L_2} \omega_K, \quad \int_{\Delta_2}\omega_K=\int_{L_1} \omega_K=\int_{L_5 - L_4}\omega_K,\\
\vspace{2mm}&\displaystyle \int_{\Delta_3} =\int_{ -L_4 - 3(L_6+L_5-L_4 -L_3+L_2+L_1)}\omega_K, \\
&\displaystyle\int_{\Delta_4}\omega_K=\int_{L_6+L_5-L_4 -L_3+L_2+L_1}\omega_K.
\end{cases}
\end{align}
\end{prop}

\begin{proof}
According to Proposition \ref{TransProp},
$\{D_1,\cdots,D_4\}$ gives a basis of ${\rm Tr}(K(X,Y))$.
Recall
the construction of $2$-cycles $\Gamma_1,\cdots,\Gamma_4$ on $S(X,Y)$ in Remark \ref{DualRemark}.
Together with Prposition \ref{equalLem}, 
it is sufficient to take $2$-cycles $\Delta_1,\cdots,\Delta_4\in {\rm H}_2(K(X,Y),\mathbb{Z})$
such that   $(\Delta_j\cdot D_k)=\delta_{jk}$.
By Corollary \ref{LDCor}, we can check that the $2$-cycles in the right hand side of (\ref{intlemma}) satisfies these properties.
\end{proof}

\subsection{The cambers $R_1,R_2,R_3$ and $R_4$}

\begin{figure}[h]
\center
\includegraphics[width=5cm,height=5cm]{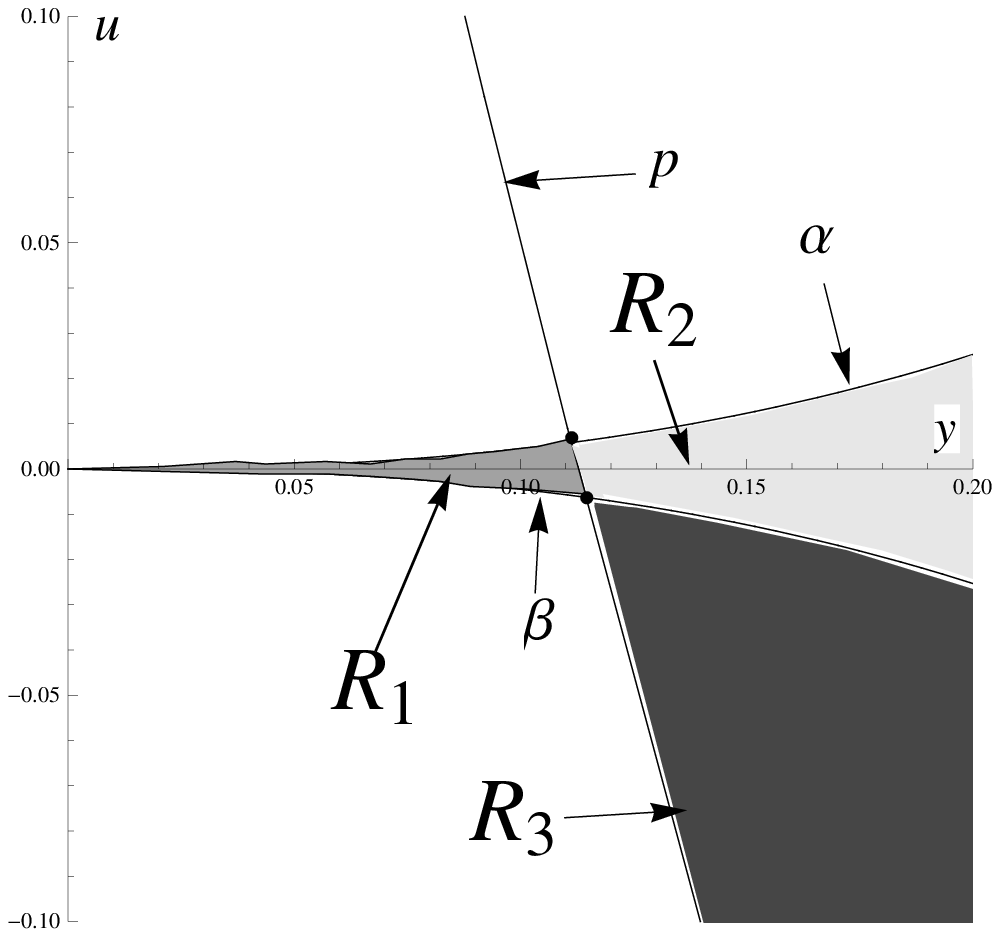}
\quad
\includegraphics[width=5cm,height=5cm]{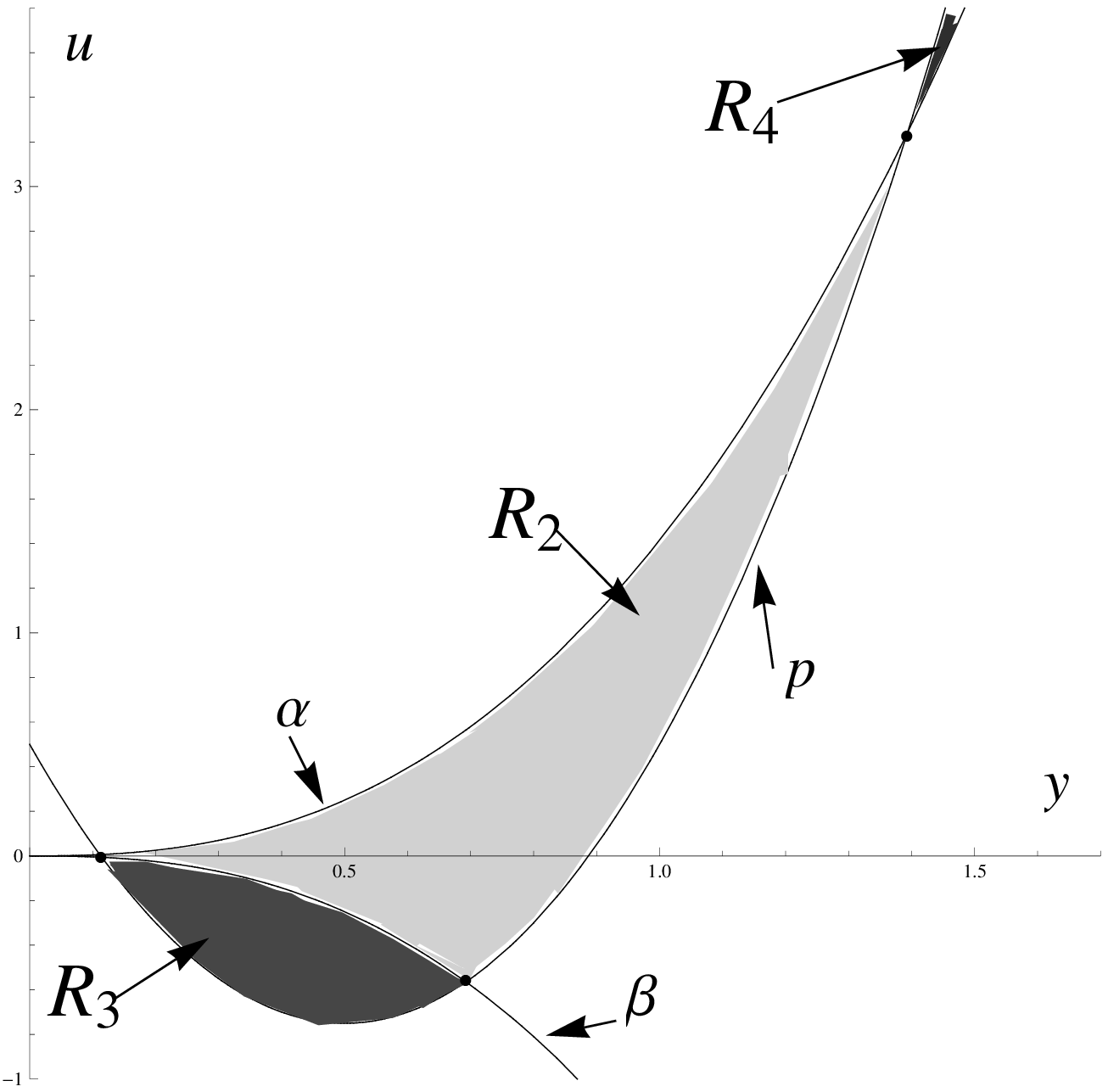}
\quad
\includegraphics[width=5cm,height=5cm]{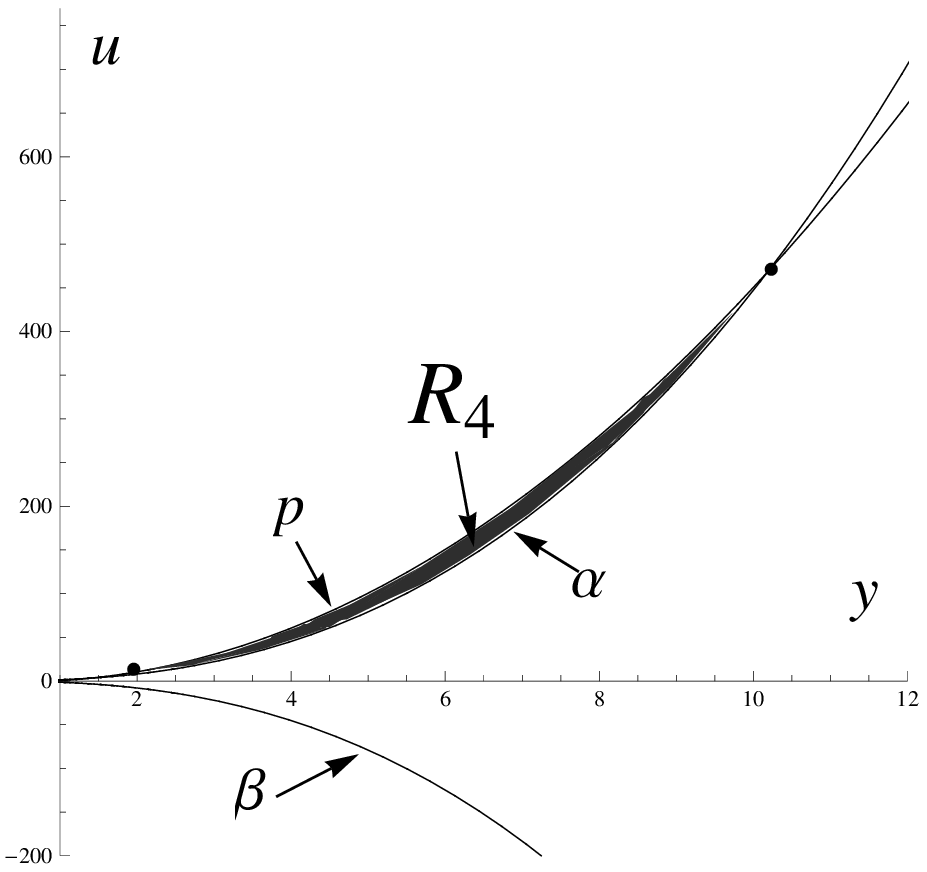}\\
{\small
$0<y<0.2$   \hspace{3cm} $0<y<1.5$ \hspace{3cm} $0<y<12$
}
\caption{The chambers $R_1,R_2,R_3$ and $R_4$.}
\end{figure}

We set the chambers in $\mathbb{R}^2$ (see Figure 9):
{\small
\begin{align}\label{R1R2R3R4}
\begin{cases}
 &R_1= \{(u,y)|0\leq y \leq s_2, w_1(y) \leq u \leq w_2(y) \}, \\
 & R_2= \{(u,y)|s_1\leq y \leq s_4, w_2(y)\leq u \leq w_3(y) \}, \\
 &R_3= \{(u,y)|s_2\leq y \leq s_3, w_1(y)\leq u \leq w_2(y)\}, \\
& R_4=\{(u,y)|s_4\leq y \leq s_5,w_2(y)\leq u \leq w_3(y) \}.
\end{cases}
\end{align}
}
They are surrounded by the branch divisors $P$ and $Q$.
From Table 2, we obtain Table 7.

\begin{table}[h]
\center
\begin{tabular}{ccc}
\toprule
$y$ & $\displaystyle \frac{1}{2}\Big(\int_{\gamma_1(y)}\omega_y\Big)$ & $\displaystyle \frac{1}{2}\Big(\int_{\gamma_2(y)}\omega_y\Big)$ \\  
  \midrule
$0<y<s_1$ & $\displaystyle \int_{\alpha(y)}^{p(y)}\frac{du}{F(u,y)}+\int_{p(y)}^{\infty}\frac{du}{F(u,y)}$& $ \displaystyle \int_{\alpha(y)}^{\beta(y)} \frac{du}{F(u,y)}  $ \\
$s_1<y<s_2$ & $\displaystyle \int_{p(y)}^{\beta(y)}\frac{du}{F(u,y)}$ & $\displaystyle \int_{\alpha(y)}^{p(y)}\frac{du}{F(u,y)}+ \int_{p(y)}^{\beta(y)}\frac{du}{F(u,y)}$  \\
$s_2<y<s_3$ & $\displaystyle \int_{\beta(y)}^{p(y)}\frac{du}{F(u,y)}$ &$\displaystyle \int_{\alpha(y)}^{\beta(y)}\frac{du}{F(u,y)}$ \\
$s_3<y<s_4$ & $ \displaystyle \int_{p(y)}^{\beta(y)}\frac{du}{F(u,y)}$ &$ \displaystyle \int_{\alpha(y)}^{p(y)}\frac{du}{F(u,y)}+ \int_{p(y)}^{\beta(y)}\frac{du}{F(u,y)}$ \\
$s_4<y<s_5$ & $\displaystyle \int_{\alpha(y)}^{p(y)}\frac{du}{F(u,y)}+\int_{p(y)}^{\infty}\frac{du}{F(u,y)}$ &$ \displaystyle \int_{p(y)}^\infty \frac{du}{F(u,y)} $  \\
$s_5<y$ & $\displaystyle  \int_{p(y)}^{\beta(y)}\frac{du}{F(u,y)}$ &$ \displaystyle \int_{\alpha(y)}^{p(y)}\frac{du}{F(u,y)}+ \int_{p(y)}^{\beta(y)}\frac{du}{F(u,y)}$ \\
\bottomrule
\end{tabular}
\caption{The elliptic integrals on $E(y)$ for $(s_{j-1},s_{j})$.}
\end{table}

\begin{thm}\label{DoubleIntegralThm}
A branch of the period mapping $\Phi_K$ in {\rm (\ref{KummerPeriod})} on $U_0$ is given by the following double integrals on the chambers $R_1,R_2,R_3$ and $R_4$:
\begin{align}\label{AlgIntegral}
\begin{cases}
&\displaystyle \int_{\Delta_1} \omega = 2 \int_{R_2} \frac{du dy}{F(u,y)} +2 \int_{R_4} \frac{du dy}{F(u,y)}, \quad\quad
\displaystyle \int_{\Delta_2} \omega =2 \int_{R_2} \frac{du dy}{F(u,y)}, \\
&\displaystyle \int_{\Delta_3} \omega =6 \int_{R_1} \frac{du dy}{F(u,y)} +2\int_{R_3}  \frac{du dy}{F(u,y)},\quad\quad
\displaystyle \int_{\Delta_4}\omega=-2\int_{R_1} \frac{du dy}{F(u,y)}.
\end{cases}
\end{align}
\end{thm}

\begin{proof}
Form Proposition (\ref{DeltaThm}), Table 6 and Table 7, 
we have
\begin{align*}
&\displaystyle \int_{\Delta_2} \omega_K = \int_{L_5} \omega_K - \int_{L_4} \omega_K\\
&\quad\quad\quad\displaystyle =2 \int_{s_1}^{s_4} \int_{\gamma_1 (y) -\gamma_2 (y)} \frac{dydu}{F(u,y)} -2\int_{s_2}^{s_3} \int_{\gamma_1 (y)}\omega_K
=2 \int_{s_1}^{s_4} \int_{p(y)}^{\alpha (y)}\frac{dy du }{F(u,y)}=\int_{R_2}\frac{dydu}{F(u,y)}.
\end{align*}
Similarly, we have 
\begin{align*}
\begin{cases}
&\displaystyle \int_{\Delta_1} \omega_K = \int_{L_5} \omega_K -\int_{L_4} \omega_K + \int_{L_2} \omega_K\\
&\quad\quad\quad\displaystyle =2 \int_{R_2} \frac{dydu}{F(u,y)} + 2\int_{s_4}^{s_5}\int_{\alpha(y)}^{p(y)}\frac{dydu}{F(u,y)}
\vspace{2mm}=2\int_{R_2}\frac{dydu}{F(u,y)} +2 \int_{R_4} \frac{dydu}{F(u,y)},\\
&\displaystyle \int_{\Delta_4}=\int_{L_6} \omega_K +\int_{L_5}\omega_K -\int_{L_4}\omega_K -\int_{L_3}\omega_K +\int_{L_2}\omega_K+\int_{L_1}\omega_K\\
&\quad\quad\quad\displaystyle =2 \int_{0}^{s_1}\int_{\beta(y)}^{\alpha(y)}\frac{dydu}{F(u,y)} + 2 \int_{s_1}^{s_2}\int_{p(y)}^{\beta(y)}\frac{dydu}{F(u,y)}
\vspace{2mm}=-\int_{R_1}\frac{dydu}{F(u,y)},\\
&\displaystyle \int_{\Delta_3}=-\int_{L_4}\omega_K - 3\int_{\Delta_4}\omega_K\\
&\quad\quad\quad\displaystyle = 2\int_{s_2}^{s_3}\int_{p(y)}^{\beta(y)}\frac{dydu}{F(u,y)} +6 \int_{R_1}\frac{dydu}{F(u,y)}
=2\int_{R_3} \frac{dydu}{F(u,y)} +6 \int_{R_1}\frac{dydu}{F(u,y)}.
\end{cases}
\end{align*}
\end{proof}

By the analytic continuation of the single-valued branch   on $U_0$ given by the integrals in (\ref{AlgIntegral}), 
we have the multivalued period mapping $\Phi_K$ for the family $\mathcal{K}$. 
Hence, the Hilbert modular functions for $\mathbb{Q}(\sqrt{5})$
is deeply concerned with the arrangement of the divisors $P$ in (\ref{P}) and $Q$ in (\ref{Q}).
The above theorem gives a canonical extension of the classical elliptic integrals to the Hilbert modular case with the smallest discriminant.

\section*{Acknowledgment}
The author would like to thank Professor Hironori Shiga for helpful advises and  valuable suggestions.
He is also grateful to Professor Kimio Ueno and the members of his laboratory for kind encouragements. 
He is grateful to the referee for careful reading and valuable comments. This work is supported by Waseda University Grant for Special Research Project 2013A - 870 and 2014B -169.

\vspace{5mm}

\begin{center}
\hspace{7.7cm}\textit{Atsuhira  Nagano}\\
\hspace{7cm}\textit{ Department of Mathematics}\\
\hspace{7.7cm}\textit{ Waseda University}\\
\hspace{7.7cm}\textit{Okubo 3-4-1, Shinjuku-ku, Tokyo, 169-8555}\\
\hspace{7.7cm}\textit{Japan}\\
 \hspace{7.7cm}\textit{(E-mail: atsuhira.nagano@gmail.com)}
  \end{center}
 \end{document}